\documentclass[12pt]{article}
\usepackage{amsmath,amsfonts,amssymb}
\usepackage{tikz}

\usepackage{graphicx}
\usepackage{caption}
\usepackage{subcaption}
\usetikzlibrary{arrows}

\usepackage{amsmath,amsfonts,ifthen,fullpage,enumerate,float}

\title{\bf A characterisation of projective unitary equivalence of 
finite frames}
\author{Tuan Chien and Shayne Waldron\\
 \\ 
Department of Mathematics \\ University of Auckland\\
Private
Bag 92019, Auckland, New Zealand\\
e--mail: waldron@math.auckland.ac.nz}
\date{\today}

\newtheorem{theorem}{Theorem}[section]
\newtheorem{lemma}[theorem]{Lemma}
\newtheorem{example}[theorem]{Example}
\newtheorem{corollary}[theorem]{Corollary}

\newtheorem{remark}[theorem]{Remark}
\newenvironment{proof}{{\noindent \bf
Proof:}}{\hfill$\Box$\bigskip}

\newcommand{\CC}{\mathbb{C}}
\newcommand{\RR}{\mathbb{R}}
\newcommand{\TT}{\mathbb{T}}
\newcommand{\FF}{\mathbb{F}}
\newcommand{\ZZ}{\mathbb{Z}}
\newcommand{\Z}{\mathbb{Z}}
\newcommand{\C}{\mathbb{C}}

\def\Cd{\CC^d}
\def\Rd{\RR^d}

\def\pmat#1{\begin{pmatrix}#1\end{pmatrix}}

\def\inpro#1{\langle#1\rangle}

\def\norm#1{\Vert#1\Vert}

\def\dep{\mathop{\rm dep}\nolimits}
\def\Aut{\mathop{\rm Aut}\nolimits}
\def\Gram{\mathop{\rm Gram}\nolimits}

\def\gr{\mathop{\rm gr}\nolimits}

\let\ga\alpha
\let\gb\beta
\let\gam\gamma

\let\gG\Gamma

\let\gth\theta
\let\gd\delta
\let\gD\Delta
\let\gl\lambda
\let\gL\Lambda
\let\gs\sigma
\def\Implies{\hskip1em\Longrightarrow\hskip1em}
\def\cB{{\cal B}}
\def\cC{{\cal C}}
\def\cH{{\cal H}}
\def\cG{{\cal G}}

\def\cS{{\cal S}}
\def\cT{{\cal T}}

\def\spam{\mathop{\rm span}\nolimits}

\def\Iff{\hskip1em\Longleftrightarrow\hskip1em}

\def\Implies{\hskip1em\Longrightarrow\hskip1em}

\newif\ifdraft\def\draft{\drafttrue\hoffset=.8truecm\showlabeltrue
\def\comment##1{{\bf comment: ##1}}
\headline={\sevenrm \hfill \ifx\filenamed\undefined\jobname\else\filenamed\fi%
(.tex) (as of \ifx\updated\undefined???\else\updated\fi)
 \TeX'ed at {\hour\time\divide\hour by 60{}%
\minutes\hour\multiply\minutes by 60{}%
\advance\time by -\minutes
\the\hour:\ifnum\time<10{}0\fi\the\time\  on \today\hfill}}
}

\begin{document}

\maketitle
\bigskip

\begin{abstract}

It is well known that two finite sequences of vectors in
inner product spaces are unitarily equivalent if and only if
their respective inner products (Gram matrices) are equal.
Here we present a corresponding result for the
projective unitary equivalence of two sequences of
vectors (lines) in inner product spaces, i.e., that  a finite 
number of (Bargmann) projective (unitary) invariants are equal.
This is based on an algorithm to recover the sequence of vectors 
(up to projective unitary equivalence) from
a small subset of
 these projective invariants.
We consider the implications for the characterisation of
SICs, MUBs and harmonic frames up to projective unitary
equivalence.
We also extend our results to projective similarity of vectors.

\end{abstract}

\bigskip
\vfill


 \ifdraft	Graphs and linear algebra (matrices, eigenvalues, etc.) \else\fi
14N05, \ifdraft	Projective techniques \else\fi
14N20, \ifdraft Algebraic geometry: Configurations and arrangements of linear subspaces \else\fi
15A83, \ifdraft Matrix completion problems \else\fi

secondary
15A04, \ifdraft	Linear transformations, semilinear transformations \else\fi
42C15, \ifdraft General harmonic expansions, frames  \else\fi
81P15, \ifdraft Quantum measurement theory \else\fi
81P45, \ifdraft	Quantum information, communication, networks \else\fi

\vskip .5 truecm
\hrule
\newpage

\section{Introduction}

Finite sequences of vectors $\Phi=(v_j)$ and $\Psi=(w_j)$ 
in (real or complex) inner product spaces  $\cH_1$ and $\cH_2$ are
{\bf unitarily equivalent} if there is a unitary map $U:\cH_1\to\cH_2$, such that
$$ w_j=Uv_j, \qquad\forall j, $$
and {\bf projectively unitarily equivalent} 
if there is a unitary map $U$ and unit scalars $c_j$, such that
$$ w_j=c_jUv_j, \qquad\forall j, $$
equivalently
$$ w_jw_j^*=U(v_jv_j^*)U^*, \qquad\forall j. $$
A finite spanning sequence of vectors for an inner product space 
is also called a {\bf finite frame}.

The {\bf Gram matrix} ({\bf Gramian}) of $\Phi=(v_j)_{j=1}^n$ is
$$ \Gram(\Phi)=[\inpro{v_k,v_j}]_{j,k=1}^n. $$
We take our inner products to be linear in the first variable.
It is well known (cf.\ \cite{H62}) that 
\begin{itemize}
\item $\Phi$ and $\Psi$ are unitarily equivalent if and only if
\begin{equation}
\Gram(\Phi)=\Gram(\Psi).  \label{Gramcdn}
\end{equation}
\item $\Phi$ and $\Psi$ are unitarily projectively equivalent if and only if
\begin{equation}
\Gram(\Psi)=C\Gram(\Phi)C^*,  
\label{projGramcdn}
\end{equation}
where $C$ is the diagonal matrix with diagonal entries $c_j$.
\end{itemize}

Clearly, (\ref{Gramcdn}) can be used to verify unitary equivalence,
as can \eqref{projGramcdn} to verify projective unitary equivalence 
for real inner product spaces (where
$c_j\in\{-1,1\}$, cf.\ Example \ref{projinRd}). 
For complex inner product spaces, 
we have no knowledge of $c_j$, other than $|c_j|=1$,
and so \eqref{projGramcdn}, i.e.,
 $$  \inpro{w_j,w_k} = c_j\overline{c_k}\inpro{v_j,v_k}, \qquad\forall j,k, $$
does not provide a practical method for verifying
projective unitary equivalence.

Following \cite{MCS01}, we define the
{\bf $m$--vertex Bargmann invariants} or {\bf $m$--products} 
of a sequence of $n$ vectors $\Phi=(v_j)$
to be
\begin{equation}
\label{mproducts}
\gD(v_{j_1},v_{j_2},\ldots,v_{j_m}):=
\inpro{v_{j_1},v_{j_2}}\inpro{v_{j_2},v_{j_3}}
\cdots\inpro{v_{j_m},v_{j_1}},
\quad 1\le j_1,\ldots, j_m\le n.
\end{equation}
In particular, (cf.\ \cite{AFF09}), we define 
the
{\bf triple products} 
to be
\begin{equation}
\label{tripeqn}
T_{jk\ell} := \gD(v_j,v_k,v_\ell)=  \inpro{v_j,v_k}\inpro{v_k,v_\ell}\inpro{v_\ell,v_j}.
\end{equation}

We observe that the $m$--products are {\em projective unitary invariants}, 
e.g., for $m=3$
\begin{align}
\gD(c_j U v_j,c_k U v_k,c_\ell U v_\ell)
&= \inpro{c_j U v_j,c_k U v_k}
\inpro{c_k U v_k,c_\ell U v_\ell}
\inpro{c_\ell U v_\ell,c_j U v_j} \nonumber  \\
& = c_j \overline{c_k} \inpro{U v_j,U v_k}
c_k \overline{c_\ell} \inpro{U v_k,U v_\ell}
c_\ell \overline{c_j} \inpro{U v_\ell,U v_j} \nonumber  \\
& = \inpro{v_j,v_k}
\inpro{v_k,v_\ell}
\inpro{v_\ell,v_j} \nonumber \\
&= \gD(v_j,v_k,v_\ell).
\label{tripareinvariant}
\end{align}
We will show (Theorem \ref{nbargmanprojunieq}) 
that a sequence of vectors $(v_j)$
is determined up to projective unitary equivalence by all its
$m$--products. 
Our proof
relies on the fact that certain small subsets of the $m$--products are 
sufficient. These depend on which of the $m$--products are nonzero,
which is conveniently encapsulated
by the {\it frame graph}.

We define the {\bf frame graph} (cf.\ \cite{AN13}) of 
a sequence of vectors $(v_j)$ to be the graph with vertices 
$\{v_j\}$ (or the indices $j$ themselves) and
$$ \hbox{an edge between $v_j$ and $v_k$, $j\ne k$} \Iff
\inpro{v_j,v_k}\ne 0. $$ 
Clearly, projectively unitarily equivalent frames have the 
same frame graph.

In Section 2, generalising the results of \cite{AFF09} for SICs,
we show that if the common frame graph of $\Phi=(v_j)$ and 
$\Psi=(w_j)$ is complete, 
then they are
 projectively unitarily equivalent 
if and only if
their $3$--products (triple products) are equal (Theorem \ref{Chartheorem}).
Later we will show this condition extends
(Theorem \ref{3-cyclechar}), e.g., to the case
when the frame graph is chordal.
We apply this result to sequences of equiangular lines
(including SICs),
then give an example to show that the $3$--products 
do not determine projective unitary equivalence in general
(Example \ref{cycleexample}).

In Sections 3 and 4, we show that projective unitary equivalence
is characterised the $m$--products
(Theorem \ref{nbargmanprojunieq}).
We show that is sufficient to consider only a small subset
of these projective invariants, which can be determined from
the frame graph (Theorem \ref{nbargmanprojunieqII}).
We give an algorithm for constructing
all sequences with given $m$--products, and
consider the classification of MUBs (Theorem \ref{MUBtheorem}).

In Section 5, we apply our results to the classification of
sequences of vectors up to (projective) similarity
(Theorem \ref{ChartheoremII}). 

Finally, in Section 6, we consider the classification of harmonic 
frames up to projective unitary equivalence (Theorem \ref{harmonicequiv}).

\section{Complete frame graphs}

A sequence of $n\ge d$ unit vectors $(v_j)$ in $\Cd$ is {\bf equiangular} if
for some $C\ge0$
$$  |\inpro{v_j,v_k}|=C, \qquad j\ne k. $$
For $C>0$, such a sequence has a complete frame graph (no zero
inner products), as does a generic sequence of vectors.
In \cite{AFF09}, it was shown that
$d^2$ equiangular vectors in $\Cd$
are characterised up to projective unitary equivalence 
by their triple products ($3$--products).
Here we modify the argument 
to when the frame graph is complete. 
We then show, by an example,
that this result does not extend to a general sequence of vectors.

The {\bf angles} of a sequence of vectors $\Phi=(v_j)$
are the $\gth_{jk}\in\TT:={\RR/(2\pi\ZZ)}$ defined by
$$  \inpro{v_j,v_k}=|\inpro{v_j,v_k}|e^{i\gth_{jk}}, 
\qquad \inpro{v_j,v_k}\ne 0.$$
Since $\inpro{v_j,v_k}=\overline{\inpro{v_k,v_j}}$, these
satisfy
$$ \gth_{jk}=-\gth_{kj}. $$
A sequence of vectors may have few angles, e.g., an orthogonal
basis has no angles.

\begin{lemma}
\label{proequivlemma}
Let $\Phi=(v_j)$ and $\Psi=(w_j)$ be finite sequences of vectors in Hilbert
spaces, with angles $\gth_{jk}$ and $\gth_{jk}'$. Then
$\Phi$ and $\Psi$ are projectively unitarily equivalent if and only if
\begin{enumerate}
\item 
Their Gramians have entries with equal moduli, i.e.,
$$ |\inpro{w_j,w_k}| = |\inpro{v_j,v_k}|, \qquad \forall j,k. $$
\item Their angles are ``gauge equivalent'', i.e., 
there exist $\phi_j\in\TT$ with.
$$ \gth_{jk}' = \gth_{jk}+\phi_j-\phi_k, \qquad \forall j,k. $$
\end{enumerate}
\end{lemma}

\begin{proof}
First suppose that $\Phi$ and $\Psi$ are projectively unitarily 
equivalent, i.e., $w_j=c_j Uv_j$, where $U$ is unitary
and $c_j=e^{i\phi_j}$. Then
\begin{align*}
e^{i\gth_{jk}'} |\inpro{w_j,w_k}|
&= \inpro{w_j,w_k}
= \inpro{c_j Uv_j,c_k Uv_k}
= c_j \overline{c_k} \inpro{Uv_j,Uv_k}
= e^{i(\phi_j-\phi_k)} \inpro{v_j,v_k} \\
&= e^{i(\phi_j-\phi_k)} e^{i\gth_{jk}} |\inpro{v_j,v_k}|
\end{align*}
By equating the moduli and then the arguments we obtain 1 and 2.

Conversely, suppose that 1 and 2 hold. 
Let $\tilde v_j := e^{i\phi_j} v_j$.
Then
\begin{align*}
\inpro{\tilde v_j,\tilde v_k}
&= \inpro{e^{i\phi_j} v_j, e^{i\phi_k}\tilde v_k}
= e^{i(\phi_j-\phi_k)} \inpro{ v_j, v_k}
= e^{i(\phi_j-\phi_k)} e^{i\gth_{jk}}  |\inpro{ v_j, v_k}| \\
&= e^{i\gth_{jk}'}  |\inpro{ w_j, w_k}|
= \inpro{ w_j, w_k}. 
\end{align*}
Thus $(w_j)$ is unitarily equivalent to $(\tilde v_j)$,
which is projectively unitarily equivalent to $(v_j)$,
and so $\Psi$ and $\Phi$ are projectively unitarily equivalent.
\end{proof}

We observe that $|\inpro{v_j,v_k}|$ can be calculated from 
the triple products of (\ref{tripeqn}), since
\begin{equation}
T_{jjj}=\inpro{v_j,v_j}^3=\norm{v_j}^6, \qquad
T_{jjk} 
=\inpro{v_j,v_j}|\inpro{v_j,v_k}|^2
=T_{jjj}^{1\over 3}|\inpro{v_j,v_k}|^2. 
\label{modulifromtriples}
\end{equation}

\begin{theorem} 
\label{Chartheorem} (Characterisation)
Let $\Phi=(v_j)_{j\in J}$ and $\Psi=(w_j)_{j\in J}$ 
be finite sequences of vectors in Hilbert 
spaces. 
Then
\begin{enumerate}
\item $\Phi$ and $\Psi$ are unitarily equivalent if and only if 
their Gramians are equal, i.e.,
$$ \inpro{v_j,v_k}=\inpro{w_j,w_k}, \qquad \forall j,k. $$
\item If the frame graphs of $\Phi$ and $\Psi$ are complete,
then they 
are projectively unitarily equivalent if and only if 
their triple products are equal, i.e.,
$$ \inpro{v_j,v_k} \inpro{v_k,v_\ell} \inpro{v_\ell,v_j}
= \inpro{w_j,w_k} \inpro{w_k,w_\ell} \inpro{w_\ell,w_j},
\qquad \forall j,k,\ell. $$
\end{enumerate}
\end{theorem}

\begin{proof}
The condition for unitary equivalence is well known. It is included
in the theorem only for the purpose of comparison. We now prove 2.

First suppose that $\Phi$ and $\Psi$ are projectively unitarily 
equivalent, i.e., $w_j=c_jUv_j$. Then by (\ref{tripareinvariant}) 
their triple products are equal.

Conversely, suppose that $\Phi$ and $\Psi$ have the same triple products,
and their common frame graph is complete, i.e., all the triple 
products are nonzero.

It follows from (\ref{modulifromtriples}) that their Gramians have 
entries with equal moduli,
i.e., 
$$ |\inpro{v_j,v_k}|=|\inpro{w_j,w_k}|, \qquad \forall j,k. $$
Let $\gth_{jk}$ and $\gth_{jk}'$ be the angles of $\Phi$ and $\Psi$.
Since the triple products have the polar form
$$ T_{jk\ell} 
= \inpro{v_j,v_k}\inpro{v_k,v_\ell}\inpro{v_\ell,v_j}
= e^{i(\gth_{jk}+\gth_{k\ell}+\gth_{\ell j})}
|\inpro{v_j,v_k}\inpro{v_k,v_\ell}\inpro{v_\ell,v_j}|, $$
we obtain
$$ \gth_{jk}+\gth_{k\ell}+\gth_{\ell j} 
= \gth_{jk}'+\gth_{k\ell}'+\gth_{\ell j}'. $$
Fix $\ell$, and rearrange this, using 
$\gth_{k\ell}=-\gth_{\ell k}$ and
$\gth_{k\ell}'=-\gth_{\ell k}'$,
to get
$$ \gth_{jk}'
= \gth_{jk} +(\gth_{\ell j} -\gth_{\ell j}') 
+ (\gth_{k\ell} -\gth_{k\ell}')
= \gth_{jk} +(\gth_{\ell j} -\gth_{\ell j}') - (\gth_{\ell k} -\gth_{\ell k}')
= \gth_{jk} +\phi_j-\phi_k,
$$
where $\phi_j:= \gth_{\ell j} -\gth_{\ell j}'$, 
i.e., the angles of $\Phi$ and $\Psi$ are gauge equivalent.
Since the conditions of 
Lemma \ref{proequivlemma} hold, it follows that $\Phi$ and $\Psi$ are
projectively unitarily equivalent.
\end{proof}

The real case is closely connected
with the theory of {\it two--graphs} (cf.\ \cite{GR01})
as follows.

\begin{example} 
\label{projinRd}
(Equiangular lines in $\Rd$).
Suppose that $\Phi=(v_j)$ is a sequence of $n>d$ equiangular unit vectors 
(lines) in $\Rd$, i.e., there is an $\ga>0$ with
$$ \inpro{v_j,v_k}=\pm \ga, \qquad j\ne k.$$
Then the Gramian matrix has the form
$$ G_\Phi=\Gram(\Phi)=I+\ga S_\Phi, $$
where $S=S_\Phi$ is a {\bf Seidel matrix}, i.e.,
$S$ is symmetric, with zero diagonal, and off
diagonal entries $\pm1$. Moreover, each Seidel matrix is 
associated with a sequence of equiangular lines.
Each Seidel matrix $S$ is in turn associated with 
the graph $\gr(S)$ which has an edge between $j\ne k$ if
and only if $S_{jk}=-1$. Let $\cC$ be the diagonal matrices with
diagonal entries $\pm1$. Then the projective unitary equivalence
class of $\Phi$ is uniquely determined by all the 
possible Gram matrices of its members, i.e.,
$$ \cG := \{ CG_\Phi C^*:C\in\cC\}, $$
and hence all the possible Seidel matrices
$$ \cS := \{ CG_\Phi C^*:C\in\cC\}, $$
and in turn the corresponding graphs $\gr(\cS)$. 
The set of graphs $\gr(\cS)$ is called 
the {\bf switching class} of $\gr(S_\Phi)$, or a {\bf two--graph}.
Since the frame graph of $\Phi$ is complete, Theorem \ref{Chartheorem}
gives that projective unitary equivalence class of $\Phi$
(equivalently $\cG$, $\cS$ or $\gr(\cS)$) is in $1$--$1$ correspondence
with the triple products of $\Phi$. It suffices to consider only those
triple products with distinct indices, since if an index is repeated 
twice or thrice, then by (\ref{modulifromtriples}) the triple 
products are depend only on $\ga$.
In this way, the two--graph is in $1$--$1$ correspondence with the 
triple products
$$ \{ T_{jk\ell}=\pm \ga^3:\hbox{$j,k,\ell$ are distinct}\}. $$
Since these triple products take only two values,
which are independent of the ordering of the indices,
they can be described by giving the collection of the subsets $\{j,k,\ell\}$
where they take one of these values.
This association leads to the equivalent definition of a two--graph
as a set of (unordered) triples chosen from a finite
vertex set $X$, such that every unordered quadruple from $X$ contains
an even number of triples of the two--graph.
\end{example}

\begin{example} (Equiangular lines in $\Cd$)
If $\Phi$ is a sequence of $n$ equiangular unit vectors (lines) in
$\Cd$, with $C>0$, then up to projective unitary equivalence $\Phi$
is determined by its triple products. This result was given in
\cite{AFF09} for the special case $n=d^2$.
Such a configuration has $C={1\over\sqrt{d+1}}$, and is known
as a {\bf SIC} or {\bf SIC-POVM} 
({symmetric informationally complete positive operator valued measure}).
\end{example}

We now give an example to show that projective unitary equivalence 
is not always characterised by the triple products if the
frame graph is not complete. 
We observe that the $m$--products are closed under 
complex conjugation, i.e.,
\begin{equation}
\overline{\gD(v_{j_1},v_{j_2},\ldots,v_{j_m})}
= \gD( v_{j_m},\ldots,v_{j_2},v_{j_1}).
\end{equation}

\begin{example}
\label{cycleexample} 
($n$--cycle)
Let $(e_j)$ be the standard basis vectors in $\CC^n$.
Fix $|z|=1$, and let
$$ v_j := 
\begin{cases}
e_j + e_{j+1}, & 1\le j<n, \\
e_n+ z e_1, & j=n.
\end{cases} $$
Then the frame graph of $(v_j)$ is the $n$--cycle
$(v_1,\ldots,v_n)$, and so the only nonzero $m$--products 
for distinct vectors are
\begin{align}
& \gD(v_j)=\norm{v_j}^2=2, \qquad 1\le j \le n,  \\
& \gD(v_j,v_{j+1})=|\inpro{v_j,v_{j+1}}|^2=1, \quad 1\le j < n,\\
& \gD(v_1,v_2,\ldots,v_n)=z,
\end{align}
and their complex conjugates. Therefore different choices of $z$ 
give projectively inequivalent frames.
Thus, for $n>3$, the
vectors $(v_j)$ are not defined up to projective unitary 
equivalence by their triple products.
\end{example}

\section{Characterisation of projective unitary equivalence} 

We now show that a sequence of $n$ vectors is determined
up to projective unitary equivalence by its $m$--products
for $1\le m\le n$. This is done by constructing a sequence of 
vectors $(w_j)$ with given $m$--products 
$\gD(v_{j_1},\ldots,v_{j_m})$,
which amounts to finding all the possible Gram matrices 
$G=[\inpro{w_k,w_j}]$
with these $m$--products, because of the following. 

\begin{remark}
Given a Gram matrix $G$, 
there are many ways to 
construct a sequence of vectors $(v_j)_{j=1}^n$
with $G=[\inpro{v_k,v_j}]$, i.e., $G=V^*V$ where $V=[v_1,\ldots,v_n]$.
For example,
since the Gram matrix $G$ of a sequence of $n$ vectors
which span a vector space of dimension $d$ is positive semidefinite
of rank $d$, it is unitarily diagonalisable
$$ G = U^* \gL U, \qquad
\gL=\pmat{\gl_1&&\cr &\ddots\cr && \gl_n}, \quad
\gl_1,\ldots,\gl_d>0, \quad \gl_{d+1},\ldots,\gl_n =0, $$
and so we may take $V=[v_1,\ldots,v_n]=\gL^{1\over 2}U$.
This gives vectors $(v_j)$ in $\C^n$ which are zero in the last $n-d$ 
components, and so can be identified with vectors in $\Cd$.
Similarly, one could take a Cholesky decomposition $G=V^*V$.
In the special case 
when $G$ is an orthogonal projection matrix $P$, and one can take
$V=P$, since $G=P=P^*P$.
\end{remark}

The diagonal entries of the Gram matrices $G=[\inpro{w_k,w_j}]$ are given by the
$1$--products, and the moduli of its off diagonal entries
by the $2$--products with distinct arguments, i.e.,
\begin{equation}
\gD(v_j)=\norm{v_j}^2,
\qquad 
\gD(v_j,v_k)=|\inpro{v_j,v_k}|^2, \quad j\ne k. 
\label{12cycles}
\end{equation}
It therefore remains to choose arguments for the 
nonzero off diagonal entries of $G$, which are consistent
with the $m$--products for $m\ge 3$.
By choosing $C$ in (\ref{projGramcdn}), some of 
these can be taken to be arbitrary. Once this is done
to the full extent (spanning tree argument), we show
the remaining arguments are then given by the $m$--products (completing
cycles).

\begin{theorem}\label{nbargmanprojunieq} (Characterisation)
Two sequences $(v_j)$ and $(w_j)$ of $n$ 
vectors are projectively unitarily equivalent
if and only if their $m$--products are equal, i.e.,
$$ \gD(v_{j_1},v_{j_2},\ldots,v_{j_m})
=\gD(w_{j_1},w_{j_2},\ldots,w_{j_m}),
\qquad 1\le j_1,\ldots, j_m\le n,
\quad 1\le m \le n. $$
\end{theorem}

\begin{proof} It suffices to find a Gram matrix 
$$G=[\inpro{w_k,w_j}]=C\Gram(\Phi)C^*$$ 
by using only the $m$--products of $\Phi=(v_j)$.
By (\ref{12cycles}), we know the modulus of each entry of $G$,
and in particular the frame graph of $\Phi$. 
We therefore need only determine the arguments of the
(nonzero) inner products, which correspond to edges of the
frame graph. This we do on each connected component $\gG$
of the frame graph of $\Phi$. 

{\it Spanning tree argument.}
Find a spanning tree $\cT$ of $\gG$ with root vertex $r$. 
By working out from the root $r$, we can multiply the vertices 
$v\in \gG\setminus\{r\}$ 
by unit scalars so that the arguments of the inner products 
corresponding to the edges of $\gG$ take arbitrarily assigned values.

{\it Completing cycles.} The only entries of the Gram matrix $G$ 
which are not yet defined are those given by the edges of $\gG$
which are not in $\cT$. Since
$\cT$ is a spanning tree, adding each such edge to $\cT$ gives an
$m$--cycle. The corresponding nonzero $m$--product has
all inner products already determined,
except the one corresponding to the added edge, 
which is therefore uniquely determined by the $m$--product.
\end{proof}

We now illustrate Theorem \ref{nbargmanprojunieq},
by constructing all  the
possible Gram matrices $G$ for a sequence of vectors
$(w_j)$ which is 
projectively unitarily equivalent to a given sequence $\Phi=(v_j)$,
by using only the $m$--products of $\Phi$.

\begin{example}
\label{example1}
Let $\Phi=(e_j)$ be an orthonormal basis for $\CC^3$,
which has Gram matrix 
$$ \Gram(\Phi)=\pmat{1&0&0\cr0&1&0\cr0 &0&1}. 
$$
Here the frame graph is totally disconnected (see Fig.\ 1), and so each possible
$G$ is determined by the $1$--products and $2$--products
using (\ref{12cycles}), i.e.,
$$ \inpro{w_j,w_j}=\inpro{v_j,v_j}=1 , \qquad
|\inpro{w_k,w_j}|^2=|\inpro{v_k,v_j}|^2=0, \quad j\ne k 
\Implies
\inpro{w_k,w_j}= \gd_{jk}. $$
Thus there is a {unique} Gram matrix $G$ corresponding 
to the $1$--products and $2$--products.
Alternatively, by (\ref{projGramcdn}), 
one has that all $G$ are given by $C^*\Gram(\Phi)C=\Gram(\Phi)$.
\end{example}

Now we give an example where
the spanning tree and cycle completing arguments are not trivial.

\begin{example}
\label{example2}
Let $\Phi=(v_j)$ be three equally spaced unit vectors in $\RR^2$,
viewed as vectors in $\CC^2$.
These have Gram matrix 
$$ \Gram(\Phi)= \pmat{1&{1\over 2}&{1\over 2}\cr
{1\over2}&1&{1\over 2}\cr
{1\over2}&{1\over2}&1}, $$
and the frame graph is complete  (see Fig.\ 1). A spanning 
tree with root $v_1$ is given by the path $v_1,v_2,v_3$,
working out from the root $v_1=w_1$, we may 
scale $v_2$ then $v_3$ to $w_j=c_j v_j$, so that 
the arguments of $\inpro{w_1,w_2}$ and $\inpro{w_2,w_3}$
are arbitrary, say $a$ and $b$, $|a|=|b|=1$. The only inner
product which is not yet determined is $\inpro{w_1,w_3}$, 
which is given by completing the $3$--cycle $v_1,v_2,v_3,v_1$, i.e.,
$$ \gD(w_1,w_2,w_3)= \gD(v_1,v_2,v_3)
\Implies \Bigl({1\over 2}a\Bigr) \Bigl({1\over 2}b\Bigr) 
\Bigl({1\over 2}\inpro{w_3,w_1}\Bigr) = \Bigl({1\over2}\Bigr)^3
\Implies \inpro{w_1,w_3}=ab. $$
Thus all the Gram matrices $G$ of vectors $(w_j)$ which are 
projectively unitarily equivalent to $\Phi=(v_j)$ are
given by 
$$ G = \pmat{1&{1\over 2}\overline{a} & {1\over 2} \overline{a}\overline{b} \cr
{1\over2}{a} & 1 &{1\over 2}\overline{b}\cr
{1\over2}{a}b & {1\over2} b &1},
\qquad |a|=|b|=1. $$
This can be checked using (\ref{projGramcdn})
$$ G= C^*\Gram(\Phi)C=
\pmat{1&{1\over 2}\overline{c_1}c_2 & {1\over 2}\overline{c_1}c_3 \cr
{1\over2}\overline{c_2}c_1 & 1 &{1\over 2}\overline{c_2}c_3\cr
{1\over2}\overline{c_3}c_1 & {1\over2}\overline{c_3}c_2 &1}. $$
\end{example}

\begin{figure}[ht]
\centering
\begin{subfigure}[b]{0.2\textwidth}
\begin{tikzpicture}[->,>=stealth',shorten >=1pt,auto,node distance=1.5cm,
  thick,main node/.style={circle,fill=white!5,draw,font=\sffamily\Large\bfseries}]

  \node[main node] (1) {};
  \node[main node] (2) [below left of=1] {};
  \node[main node] (3) [below right of=1] {};
\end{tikzpicture}
\end{subfigure}
\qquad\qquad
\begin{subfigure}[b]{0.2\textwidth}
\begin{tikzpicture}[>=stealth',shorten >=1pt,auto,node distance=1.5cm,
  thick,main node/.style={circle,fill=white!5,draw,font=\sffamily\Large\bfseries}]

  \node[main node] (1) {};
  \node[main node] (2) [below left of=1] {};
  \node[main node] (3) [below right of=1] {};

  \path[every node/.style={font=\sffamily\small}]
    (1) edge node [left] {} (2)
    (2) edge node [left] {} (3)
    (3) edge node [right] {} (1);
\end{tikzpicture}
\end{subfigure}
\caption{The frame graph of an orthonormal basis for $\C^3$ 
(Example \ref{example1}), and the frame graph for three equiangular vectors
in $\C^2$ (Example \ref{example2}).  }
\end{figure}
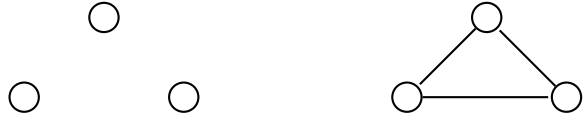

\begin{example}
\label{examplemub} 
Let $\Phi=(v_j)$ be the ``two mutually unbiased bases'' for $\CC^2$ given by
$$ \Phi=\left\{ \pmat{1\cr0},\pmat{0\cr1},
\pmat{{1\over\sqrt{2}}\cr{1\over\sqrt{2}}},
\pmat{{1\over\sqrt{2}}\cr-{1\over\sqrt{2}}} \right\},
\qquad \Gram(\Phi)=\pmat{1&0&{1\over\sqrt{2}}&{1\over\sqrt{2}}\cr
0&1&{1\over\sqrt{2}}&-{1\over\sqrt{2}}\cr
{1\over\sqrt{2}}&{1\over\sqrt{2}}&1&0\cr
{1\over\sqrt{2}}&-{1\over\sqrt{2}}&0&1\cr }
. $$
This $\Phi$ has frame graph the $4$--cycle
$v_1,v_3,v_2,v_4,v_1$, and hence no nonzero triple products with distinct indices.
A spanning tree with root $v_1$ is given by the path
$v_1,v_3,v_2,v_4$. Fix $v_1=w_1$, and then scale
in order $v_3$, $v_2$, $v_4$ to $w_j=c_jv_j$, so
that 
$$\inpro{w_1,w_3}={a\over\sqrt{2}}, \qquad
\inpro{w_3,w_2}={\overline{b}\over\sqrt{2}}, \qquad 
\inpro{w_2,w_4}={c\over\sqrt{2}}. $$
Then $\inpro{w_4,w_1}={1\over\sqrt{2}}\overline{z}$ is determined by completing
the $4$--cycle, i.e.,
$$ \inpro{w_1,w_3}\inpro{w_3,w_2}\inpro{w_2,w_4}\inpro{w_4,w_1} 
= \inpro{v_1,v_3}\inpro{v_3,v_2}\inpro{v_2,v_4}  \inpro{v_4,v_1}
\Implies 
a\overline{b}c\overline{z}=-1 $$
Thus the Gram matrices which match all the $m$--products of $\Phi$ 
have the form
$$ G = \pmat{1&0&{\overline{a}\over\sqrt{2}}&{\overline{z}\over\sqrt{2}}\cr
0&1&{\overline{b}\over\sqrt{2}}&{\overline{c}\over\sqrt{2}}\cr
{{a}\over\sqrt{2}}&{{b}\over\sqrt{2}}&1&0\cr
{{z}\over\sqrt{2}}&{{c}\over\sqrt{2}}&0&1\cr }, \qquad
|a|=|b|=|c|=1, \quad z:= -{ac\over b}.  $$ 
In this example, $\Phi$ is in fact determined up to projective unitary 
equivalence by its $1$--products and $2$--products, since Sylvester's
criterion for the $G$ above to be positive semidefinite gives
$$ \det(G)
=-{1\over 4} {(bz+ac)^2\over abcz}
=-{1\over 4} \left|{bz\over ac}+1\right|^2
\ge0
\Implies {bz\over ac}+1=0
\Implies z=-{ac\over b}. $$
By way of contrast, the $\Phi$ with frame graph a $4$--cycle in
Example \ref{cycleexample} is not determined up to  projective unitary
equivalence by its $1$--products and $2$--products.
\end{example}

\section{Reconstruction from the $m$--products}

It is apparent from the proof of Theorem \ref{nbargmanprojunieq}  
that only 
a small subset of the $m$--products is required to 
determine a sequence of vectors $\Phi$ up to projective
unitary equivalence. 
We call a subset of the $m$--products (or the corresponding indices)
a {\bf determining set} for the $m$--products if all $m$--products
can be determined from them.

\begin{corollary}
Let $\gG$ be the frame graph of a sequence of $n$ vectors $\Phi$
(this is determined by the $2$--products).
For each connected component $\gG_j$ of $\gG$, let $\cT_j$ be a spanning tree.
Then $\Phi$ is determined up to projective unitary equivalence
by the following $m$--products
\begin{enumerate}[(i)]
\item The $2$--products
\item The $m$--products, $3\le m\le n$,
used to obtain $\gG_j$ from $\cT_j$ 
by completing $m$--cycles (these have indices in $\gG_j$), 
 as detailed in 
the proof of Theorem \ref{nbargmanprojunieq}.
\end{enumerate}
In particular, if $M$ is the number of edges of $\gG$
which are not in any $\cT_j$, 
then it is sufficient to know all of 
the $2$--products, and $M$ of the $m$--products, $3\le m\le n$.
\end{corollary}

In other words, all possible $m$--products can be determined from those of
(i) and (ii),
which therefore are a determining set.
As indicated by Example \ref{examplemub}, the $m$--products of (ii)
may not be a minimal such subset, though it is a minimal subset
from which the $m$--products can be determined using only 
the proof of Theorem \ref{nbargmanprojunieq}.

\begin{example}
\label{TpTsExample}
Let $\Phi=(v_j)$ be four equiangular vectors with $C>0$. 
The frame graph of $\Phi$ is complete,
and $M=6-3=3$.
Spanning trees (see Fig.\ 2) include
\begin{align*}
\cT_p &:= \hbox{the path $v_1,v_2,v_3,v_4$},  \\
\cT_s &:= \hbox{the star graph internal vertex $v_1$ and leaves $v_2,v_3,v_4$}.
\end{align*}
For $\cT_p$
we can complete either the
$4$--cycle $v_1,v_2,v_3,v_4,v_1$ followed by
any one of the four $3$--cycles (which also completes 
another $3$--cycle) and then one of the remaining two $3$--cycles,
so, e.g., a determining set is given by the $2$--products, and
\begin{equation}
\label{Tpexample}
\gD(v_1,v_2,v_3,v_4), \quad
\gD(v_1,v_2,v_3), \quad 
\gD(v_1,v_2,v_4). 
\end{equation}
For $\cT_s$, adding any edge completes a $3$--cycle
containing $v_1$, there are then two edges which can be
added to complete a $3$--cycle, after which adding the other edge
completes the remaining two $3$--cycles,
so, e.g., a determining set is given by the $2$--products, and
\begin{equation}
\label{Tsexample}
\gD(v_1,v_2,v_3), \quad
\gD(v_1,v_2,v_4), \quad 
\gD(v_1,v_3,v_4). 
\end{equation}
\end{example}

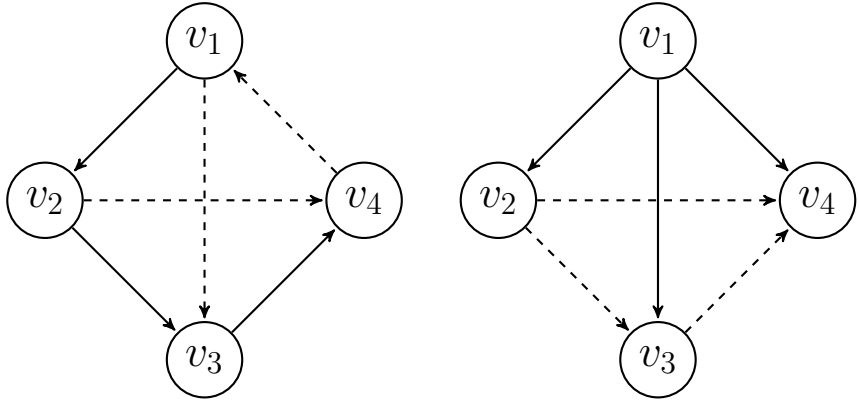
\begin{figure}[ht]
\centering
\begin{subfigure}[b]{0.34\textwidth}
\begin{tikzpicture}[->,>=stealth',shorten >=1pt,auto,node distance=3cm,
  thick,main node/.style={circle,fill=white!20,draw,font=\sffamily\Large\bfseries}]

  \node[main node] (1) {$v_1$};
  \node[main node] (2) [below left of=1] {$v_2$};
  \node[main node] (3) [below right of=2] {$v_3$};
  \node[main node] (4) [below right of=1] {$v_4$};

  \path[every node/.style={font=\sffamily\small}]
    (1) edge node [left] {} (2)
    (2) edge node [left] {} (3)
    (3) edge node [right] {} (4);

  \path[every node/.style={font=\sffamily\small}, dashed]
    (4) edge node [right] {} (1)
    (1) edge node [left, pos=0.3] {} (3)
    (2) edge node [below, pos=0.7] {} (4);
\end{tikzpicture}
\end{subfigure}
~
\begin{subfigure}[b]{0.34\textwidth}
 
\begin{tikzpicture}[->,>=stealth',shorten >=1pt,auto,node distance=3cm,
  thick,main node/.style={circle,fill=white!20,draw,font=\sffamily\Large\bfseries}]

  \node[main node] (1) {$v_1$};
  \node[main node] (2) [below left of=1] {$v_2$};
  \node[main node] (3) [below right of=2] {$v_3$};
  \node[main node] (4) [below right of=1] {$v_4$};

  \path[every node/.style={font=\sffamily\small}]
    (1) edge node [left] {} (2)
    (1) edge node [left,pos=0.3] {} (3)
    (1) edge node [right] {} (4);

  \path[every node/.style={font=\sffamily\small}, dashed]
    (2) edge node [left] {} (3)
    (3) edge node [right] {} (4)
    (2) edge node [below, pos=0.7] {} (4);
\end{tikzpicture}
\end{subfigure}

\caption{The spanning trees $\cT_p$ and $\cT_s$ (and cycle completions)
of Example \ref{TpTsExample}.  }
\end{figure}

We observe that the $4$--product of (\ref{Tpexample}) can be 
``decomposed'' into smaller $m$--products, e.g.,
\begin{equation}
\label{decompex}
\gD(v_1,v_2,v_3,v_4) = 
{\gD(v_1,v_2,v_3)\gD(v_1,v_3,v_4)\over\gD(v_1,v_3)},
\end{equation}
and so $\gD(v_1,v_2,v_3,v_4)$ can be replaced by 
the $3$--cycle $\gD(v_1,v_3,v_4)$, which gives the determining
set of (\ref{Tsexample}). 

The $m$--products of (ii) can be taken to be those
corresponding to the {\bf fundamental cycles},
i.e., the unique cycle completed by adding an edge to $\gG_j$.
These fundamental cycles form a basis for the {\bf cycle space},
i.e., the $\ZZ_2$--formal combinations of cycles.
Using this terminology, we can generalise Theorem \ref{nbargmanprojunieqII}.

\begin{theorem}\label{nbargmanprojunieqII} (Characterisation II)
Two finite sequences $(v_j)_{j\in J}$ and $(w_j)_{j\in J}$ of
vectors are projectively unitarily equivalent
if and only if their $2$--products are equal
and their $m$--products corresponding to a basis
for the cycle space of their frame graph are equal.
\end{theorem}

\begin{proof} All cycles in the frame graph can be calculated 
from those in basis, and hence all $m$--products can be 
calculated from those corresponding to a basis, as
indicated in (\ref{decompex}). 
\end{proof}

\begin{example}
The fundamental cycles corresponding to the trees 
$\cT_p$ and $\cT_s$ of Example \ref{TpTsExample} give
the following $m$--products
\begin{equation}
\gD(v_1,v_2,v_3,v_4), \quad
\gD(v_1,v_2,v_3), \quad 
\gD(v_2,v_3,v_4),
\end{equation}
\begin{equation}
\gD(v_1,v_2,v_3), \quad
\gD(v_1,v_2,v_4), \quad 
\gD(v_1,v_3,v_4), 
\end{equation}
respectively.
\end{example}

We can now generalise Theorem \ref{Chartheorem}.

\begin{theorem}
\label{3-cyclechar}
Let $\Phi=(v_j)$ be finite sequence of vectors.
Then $\Phi$ is determined up to projective unitary equivalence
by its $3$--products if the cycle space of its frame graph is
spanned by $3$--cycles (and so there is a basis of $3$--cycles).
\end{theorem}

\begin{example} (Chordal graphs)
A graph is said to be {\bf chordal}  (or {\bf triangulated})
if each of its cycles of four or more vertices has a chord,
and so the cycle space is spanned by the $3$--cycles.
Thus, if the frame graph of $\Phi$ is chordal,
as is the case for equiangular lines, then $\Phi$
is determined by its triple products. 
The extreme cases are a totally disconnected graph (orthogonal 
bases) where there are no cycles, and the complete graph where all
subsets of three vectors lie on a $3$-cycle.
\end{example}

We now give an example (Theorem \ref{MUBtheorem}) where 
the cycle space of the frame graph has a basis of $3$--cycles,
but is not chordal.  

A family of orthonormal bases $\cB_1,\cB_2,\ldots,\cB_k$ 
for $\C^d$ is said to be {\bf mutually unbiased} if
$$ |\inpro{v,w}|^2 = \frac{1}{d}, \qquad v\in B_r,\quad w \in B_j. $$
We call $\cB_1,\ldots,\cB_k$ a sequence of $k$ {\bf MUBs} 
({\bf mutually unbiased bases}).
The frame graph of two or more MUBs ($d>1$) is not chordal, 
because there is a $4$--cycle $(v_1,w_1,v_2,w_2)$, 
$v_1,v_2\in\cB_r$, $w_1,w_2\in\cB_s$ not containing a chord.

We now show for three or more MUBs the cycle space
of the frame graph is spanned by the $3$--cycles. 
This is not case for two MUBs
(cf.\ Example \ref{examplemub}).

\begin{theorem}
\label{MUBtheorem} (MUBs) Let $\Phi$ consist of three or more MUBs in $\Cd$.
Then $\Phi$ is determined up to projective unitary equivalence 
by its $3$--products.
\end{theorem}

\begin{proof}
It suffices to show the cycle space of the frame graph $\gG$ of $\Phi$ 
has a basis of $3$--cycles.
To this end, let
$\cB_j$,
$j=1,\ldots,k$, be the orthonormal bases for $\Cd$,
so that $\gG$ is a complete $k$--partite graph
(with partite sets $\cB_j$).
Fix $v_1\in\cB_1$ and $v_2\in\cB_2$.
A spanning tree $\cT$ for $\gG$ is given by taking an edge
from $v_1$ to each vertex of $\cB_j$, $j\ne 1$, and 
an edge from $v_2$ to each vertex of $\cB_1\setminus v_1$.
Each of the remaining 
edges of $\gG\setminus\cT$ 
gives a fundamental cycle. These have two types
(see Fig.\ 3):
\begin{enumerate}
\item ${1\over 2}d^2(k-1)(k-2)$ edges between vertices in $\cB_r$ and $\cB_s$, $r,s\ne1$,
which give fundamental $3$--cycles (involving $v_1$).
\item $(d-1)((k-1)d-1)$ edges between vertices $u\in\cB_1\setminus v_1$ and
$w\in\cup_{j\ne1}\cB_j\setminus v_2$,
which give fundamental $4$--cycles $(u,w,v_1,v_2)$.
These can be written as a union of the $3$--cycles
$(u,w,v_2)$ and $(v_1,v_2,w)$.
\end{enumerate}
Thus the cycle space is spanned by $3$--cycles.
\end{proof}

\begin{figure}[ht]
\centering
\begin{subfigure}[b]{0.34\textwidth}
\begin{tikzpicture}[scale=2.5]
\def \n {9}
\def \circleangle {360/3}
\tikzstyle{every node}=[draw,shape=circle];
\path (0:1cm) node (1) {};
\path (20:1cm) node (2) {};
\path (40:1cm) node (3) {};
\path ({\circleangle}:1cm) node (4) {};
\path ({20+\circleangle}:1cm) node (5) {};
\path ({40+\circleangle}:1cm) node (6) {};
\path ({2*\circleangle}:1cm) node (7) {};
\path ({20+2*\circleangle}:1cm) node (8) {};
\path ({40+2*\circleangle}:1cm) node (9) {};

\foreach \x in {1,...,3}{%
        \foreach \y in {4,...,\n}{%
                \draw (\x) -- (\y);
        }
}

\foreach \x in {4,...,6}{%
        \foreach \y in {7,...,\n}{%
                \draw (\x) -- (\y);
        }
}
\end{tikzpicture}
\end{subfigure}
\qquad
\begin{subfigure}[b]{0.34\textwidth}
\begin{tikzpicture}[scale=2.5]
\def \n {3}
\def \circleangle {360/\n}
\tikzstyle{every node}=[draw,shape=circle];
\path (0:1cm) node (1) {$v_1$};
\path (20:1cm) node (2) {};
\path (40:1cm) node (3) {};
\path ({\circleangle}:1cm) node (4) {$v_2$};
\path ({20+\circleangle}:1cm) node (5) {};
\path ({40+\circleangle}:1cm) node (6) {};
\path ({2*\circleangle}:1cm) node (7) {};
\path ({20+2*\circleangle}:1cm) node (8) {};
\path ({40+2*\circleangle}:1cm) node (9) {};
\draw (1) -- (4)
(1) -- (5)
(1) -- (6)
(1) -- (7)
(1) -- (8)
(1) -- (9)
(4) -- (2)
(4) -- (3);
\end{tikzpicture}
\end{subfigure}
\\
\begin{subfigure}[b]{0.34\textwidth}
\begin{tikzpicture}[scale=2.5]
\def \n {3}
\def \circleangle {360/\n}
\tikzstyle{every node}=[draw,shape=circle];
\path (0:1cm) node (1) {$v_1$};
\path (20:1cm) node (2) {};
\path (40:1cm) node (3) {};
\path ({\circleangle}:1cm) node (4) {};
\path ({20+\circleangle}:1cm) node (5) {};
\path ({40+\circleangle}:1cm) node (6) {};
\path ({2*\circleangle}:1cm) node (7) {};
\path ({20+2*\circleangle}:1cm) node (8) {};
\path ({40+2*\circleangle}:1cm) node (9) {};
\draw (1) -- (4)
(4) -- (8)
(8) -- (1);
\end{tikzpicture}
\end{subfigure}
\qquad
\begin{subfigure}[b]{0.34\textwidth}
\begin{tikzpicture}[scale=2.5]
\def \n {3}
\def \circleangle {360/\n}
\tikzstyle{every node}=[draw,shape=circle];
\path (0:1cm) node (1) {$v_1$};
\path (20:1cm) node (2) {};
\path (40:1cm) node (3) {$u$};
\path ({\circleangle}:1cm) node (4) {$v_2$};
\path ({20+\circleangle}:1cm) node (5) {};
\path ({40+\circleangle}:1cm) node (6) {};
\path ({2*\circleangle}:1cm) node (7) {};
\path ({20+2*\circleangle}:1cm) node (8) {$w$};
\path ({40+2*\circleangle}:1cm) node (9) {};
\draw (1) -- (4)
(4) -- (3)
(3) -- (8)
(8) -- (1);
\end{tikzpicture}
\end{subfigure}
\caption{The proof of Theorem \ref{MUBtheorem} for
MUBs $\cB_1,\cB_2,\cB_3$ in $\C^3$.
The frame graph $\gG$, the spanning tree $\cT$, and fundamental 
cycles of type $1$ and $2$.}
\end{figure}
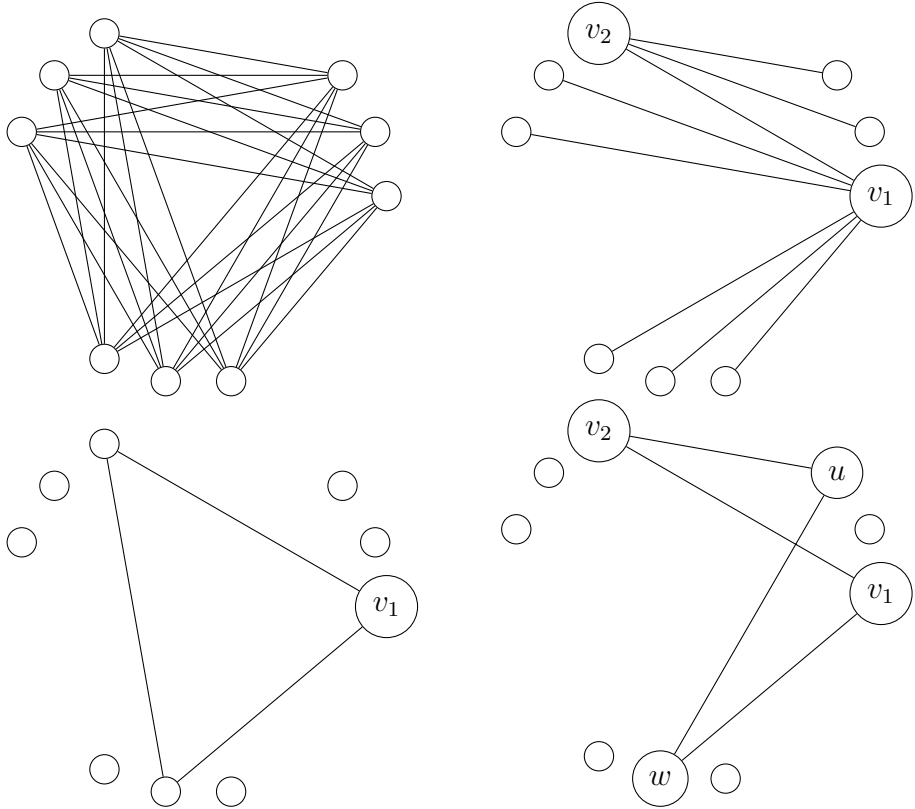

The maximal number of MUBs is of interest in
quantum information theory.
For $d$ a prime, or a power of a prime, the maximal number of MUBs 
in $\Cd$ is $d+1$, see 
\cite{bengtsson2006three}, 
\cite{ivonovic1981geometrical}, 
\cite{wootters1989optimal} for constructions.
These have a special (Heisenberg) structure, which has been
used to classify them up to projective unitary equivalence,
see \cite{kantor2012mubs}, \cite{godsil2009equiangular},
\cite{bengtsson2006three}. Our classification using $3$--products
does not presuppose any structure on the MUBs.

There exists graphs which are not chordal, with
every edge on a $3$--cycle (as is the case for the
frame graph of three or more MUBs), but for which
the cycle space is not spanned by $3$--cycles 
(see Fig.\ 4).

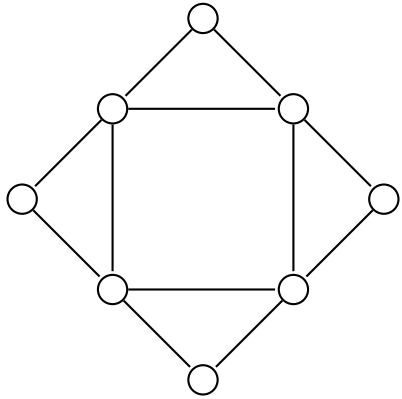
\begin{figure}[ht]
\centering
\begin{tikzpicture}[>=stealth',shorten >=1pt,auto,node distance=1.7cm,
  thick,main node/.style={circle,fill=white!1,draw,font=\sffamily\Large\bfseries}]

  \node[main node] (1) {};
  \node[main node] (2) [below left of=1] {};
  \node[main node] (3) [below right of=1] {};
  \node[main node] (4) [below left of=2] {};
  \node[main node] (5) [below right of=3] {};
  \node[main node] (6) [below right of=4] {};
  \node[main node] (7) [below left of=5] {};
  \node[main node] (8) [below right of=6] {};

  \path[every node/.style={font=\sffamily\small}]
    (1) edge node {} (2)
    (1) edge node {} (3)
    (2) edge node {} (3)
    (2) edge node {} (4)
    (2) edge node {} (6)
    (3) edge node {} (5)
    (3) edge node {} (7)
    (4) edge node {} (6)
    (5) edge node {} (7)
    (6) edge node {} (7)
    (6) edge node {} (8)
    (7) edge node {} (8);
\end{tikzpicture}
\caption{A nonchordal graph for which each edge is on a $3$--cycle.}
\end{figure}

\section{Similarity and $m$--products for vector spaces}

Using the theory of frames for vector spaces \cite{W11}, 
one can give analogous results for vector spaces, where 
the role of unitary equivalence is played by ``similarity'',
and the role of $m$--products by ``canonical $m$--products''.
This allows the ``projective symmetry group'' to be defined
in a very general setting (see \cite{CW14}).

Let $\Phi=(v_j)$ and $\Psi=(w_j)$ be finite sequences of vectors which span
vector spaces  $X$ and $Y$ 
over a subfield $\FF$ of $\CC$. 
We say that $\Phi$ and $\Psi$ are {\bf similar} if there is an invertible
linear map $Q:X\to Y$ with
$$ w_j=Qv_j, \qquad \forall j, $$
and
{\bf projectively similar} if there is an invertible
linear map $Q:X\to Y$ and unit scalars $c_j$ with
$$ w_j= c_jQv_j, \qquad \forall j. $$
For a finite sequence $\Phi=(v_j)_{j\in J}$ in $X$ the 
{\bf synthesis map} is
$$ V=[v_j]_{j\in J}:\FF^J\to X: a\mapsto \sum_j a_j v_j. $$
The subspace of all linear dependencies between the vectors 
of $\Phi$ is
$$ \dep(\Phi):=\ker(V)
=\{ a\in\FF^J: \hbox{$\sum _j a_jv_j=0$}\}, $$
and we denote the orthogonal projection onto 
$\dep(\Phi)^\perp$ (orthogonal complement) by $P_\Phi$.

We have following characterisation of similarity
in terms of linear dependencies.

\begin{lemma} (\cite{W11})
\label{lindeplemma}
Let $\Phi=(v_j)_{j\in J}$ and $\Psi=(w_j)_{j\in J}$ be spanning sequences
for the $\FF$--vector spaces $X$ and $Y$.
Then the following are equivalent
\begin{enumerate}[(a)]
\item $\Phi$ and $\Psi$ are similar, i.e.,
there is a invertible linear map $Q:v_j\mapsto w_j$.
\item $\dep(\Phi)=\dep(\Psi)$ (the dependencies are equal).
\item $P_\Phi=P_\Psi$ (the associated projections are equal).
\end{enumerate}
\end{lemma}

The proof of Lemma \ref{lindeplemma} shows that $\Phi=(v_j)$ is similar to 
columns of $P=P_\Phi$. These columns $(P e_j)$ span a subspace of $\FF^J$,
which inherits the Euclidean inner product. Indeed
$$ \inpro{Pe_j,Pe_k} = P_{jk}, $$
i.e., the Gramian of $(Pe_j)$ is $P=P_\Phi$.
We will call the $m$--products of $(Pe_j)$ the
{\bf canonical $m$--products} of $(v_j)$, and denote them
\begin{equation}
\label{canmprodef}
\gD_C(v_{j_1},\ldots,v_{j_m}) 
:= \gD(Pe_{j_1},\ldots,Pe_{j_m})
= P_{j_1j_2}P_{j_2j_3}\cdots P_{j_mj_1}. 
\end{equation}

In this way, we may apply Theorem  \ref{Chartheorem}.

\begin{theorem}
\label{ChartheoremII} (Characterisation)
Let $\Phi=(v_j)$ and $\Psi=(w_j)$ be finite sequences of vectors in 
vector spaces over a subfield $\FF$ of $\CC$ which is
closed under complex conjugation.
Then
\begin{enumerate}
\item $\Phi$ and $\Psi$ are similar if and only if $P_\Phi=P_\Psi$.
\item $\Phi$ and $\Psi$ are projectively similar if and only if
their canonical $m$--products (for a determining set) are equal.
\end{enumerate}
\end{theorem}

\begin{proof} 
The first follows from Lemma \ref{lindeplemma},
and implies that $\Phi$ and $\Psi$ are projectively similar, i.e.,
$$ w_j= c_j Q v_j = Q(c_j v_j), \qquad\forall j, $$
if and only if $\Psi=(w_j)$ and $\Phi'=(c_j v_j)$ are similar,
for some choice of unit scalars $(c_j)$, i.e.,
\begin{equation}
\label{similarityeqn}
P_\Psi = P_{(c_jv_j)} = C^* P_\Phi C.
\end{equation}
Here the last equality follows by a simple calculation.
Since $\Phi$ and $\Psi$ are similar to $(P_\Phi e_j)$ and 
$(P_\Psi e_j)$,
which have Gram matrices $P_\Phi$ and $P_\Psi$,
it follows from  (1.2) that (\ref{similarityeqn}) is equivalent
to $(P_\Phi e_j)$ and $(P_\Psi e_j)$ being projectively unitary
equivalent, 
and by Theorem \ref{Chartheorem},
this is equivalent to their $m$--products,
i.e., the canonical $m$--products of $(v_j)$ and $(w_j)$ being equal.
\end{proof}

For the case of projective similarity, 
one can calculate the $c_j$ and $Q$ in $w_j=c_j Qv_J$ explicitly,
as we now illustrate.

\begin{example}
Suppose that $\Phi=(v_j)$ spans a $2$--dimensional space,
i.e.,
$$ \ga v_1+\gb v_2+\gam v_3=0, \qquad |\ga|^2+|\gb|^2+|\gam|^2=1. $$
Then $\dep(\Phi)=\spam\{u\}$, $u=(\ga,\gb,\gam)$, so that
$$ P_\Phi= I-uu^*
= \pmat{ 
1-|\ga|^2 & -\ga\overline{\gb} & -\ga\overline{\gam} \cr
-\overline{\ga}\gb & 1-|\gb|^2 & -\gb\overline{\gam} \cr
-\overline{\ga}\gam & -\overline{\gb}\gam & 1-|\gam|^2 \cr
}. $$
The canonical $2$--products are uniquely determined by
$|a|,|b|,|c|$, e.g.,
$$ \gD_C(v_1,v_1) = (1-|\ga|^2)^2, \qquad
\gD_C(v_1,v_2) = |-\overline{\ga}\gb|^2= |\ga|^2|\gb|^2, $$
as is the canonical $3$--product corresponding to the unique
$3$--cycle
$$\gD_C(v_1,v_2,v_3) 
= (-\overline{\ga}\gb)(-\overline{\gb}\gam)(-\ga\overline{\gam})
= -|\ga|^2|\gb|^2|\gam|^2. $$
Thus if $\Psi=(w_j)$ is given by
$$ \tilde\ga w_1+\tilde\gb w_2+\tilde\gam w_3=0, 
\qquad |\tilde\ga|^2+|\tilde\gb|^2+|\tilde\gam|^2=1, $$
then
\begin{enumerate}
\item $\Phi$ is similar to $\Phi$ if and only if
$P_\Psi=P_\Phi$, i.e.,
$$\tilde\ga\overline{\tilde\gb}=\ga\overline{\gb},\qquad
\tilde\ga\overline{\tilde\gam}=\ga\overline{\gam},\qquad
\tilde\gb\overline{\tilde\gam}=\gb\overline{\gam}. $$
\item $\Phi$ is projectively similar to $\Phi$ if and only if
their canonical $m$--products are equal, 
i.e., 
$$|\tilde\ga|=|\ga|, \qquad |\tilde\gb|=|\gb|,\qquad |\tilde\gam|=|\gam|. $$
\end{enumerate}
When $\Psi$ and $\Phi$ are projectively similar,
i.e., $w_j=c_jQv_j$ (the canonical
$m$--products are equal), one has $P_\Psi=C^*P_\Phi C$.
Here, suppose $\ga,\gb,\gam\ne0$, then we have
$$ \overline{c_1}c_2\ga\overline{\gb}= \tilde\ga\overline{\tilde\gb}, \quad
\overline{c_1}c_3\ga\overline{\gam}= \tilde\ga\overline{\tilde\gam}, \quad
\overline{c_2}c_3\gb\overline{\gam}= \tilde\gb\overline{\tilde\gam}
\Implies
c_2 = {\tilde\ga\over\ga}{\gb\over\tilde{\gb}}c_1, \quad
c_3 = {\tilde\ga\over\ga}{\gam\over\tilde{\gam}}c_1.
$$
where $Q$ is defined by
$$ Qv_1 := \overline{c_1}w_1, \qquad
Qv_2:=\overline{c_2}w_2
= {\ga\over\tilde \ga}{\tilde\gb\over{\gb}}\overline{c_1} w_2. $$
\end{example}

\section{Projectively equivalent harmonic frames}

Orthogonal bases can be generalised as follows
(cf.\ \cite{CK13}, \cite{W15}).
We say, a sequence of $n$ vectors $(v_j)$ is a 
{\bf tight frame} for 
a $d$--dimensional inner product space $\cH$ if for some $A>0$
$$ f = {1\over A} \sum_j \inpro{f,v_j}v_j, \qquad\forall f\in\cH, $$
Examples of tight frames of more than $d$ vectors include SICs,
MUBs, and harmonic frames.

Let $G$ be a finite abelian group of order $n$,
with irreducible characters $\xi\in\hat G$.
Here $\hat G$ is known as the character group
(which is isomorphic to $G$).
Let $J\subset\hat G$, with $|J|=d$,
then any tight frame which is unitarily equivalent to the equal--norm
tight frame for $\C^J\approx\C^d$ given by
$$ \Phi_J=(\xi|_J)_{\xi\in\hat G} $$
is called a {\bf harmonic frame}, and 
is said to be {\bf cyclic} if $G$ is a cyclic group.
This is the class of tight frames which are the orbit
of a group of unitary transformations on $\Cd$, which is
isomorphic to $G$ (see \cite{VW05},\cite{CW11}).
The harmonic frames were studied up to unitary
equivalence in \cite{HW06}, \cite{CW11}.
We now recount some of the basic details.

Let $G$ be a fixed finite abelian group.
Subsets $J$ and $K$ of $G$ are {\bf multiplicatively 
equivalent} if there is an automorphism $\gs:G\to G$ for which
$K=\gs J$. In this case, 
$$\hat \gs:\hat G\to\hat G:
\chi\mapsto \chi\circ\gs^{-1}$$
is an automorphism of $\hat G$, and
$$\inpro{\xi|_J,\eta|_J} = \inpro{\hat\gs\xi|_K,\hat\gs\eta|_K},$$
i.e., $\Phi_J$ and $\Phi_K$ are unitarily equivalent
after reindexing by the automorphism $\hat\gs$.

The {\bf translations} of $G$ are the bijections
$$ \tau_b : G\to G:j\mapsto j-b, \qquad b\in G, $$
and we say $K$ is a {\bf translate} of $J$ if
$K=J-b$, i.e., $K=\tau_b J$.
We define the {\bf affine group} of $G$ to be the group 
of bijections $\pi:G\to G$
generated by the translations and automorphisms, i.e., 
the $|G|\, |\Aut(G)|$ maps of the
form
$$ \pi(g) =\gs(g)-b, \qquad \gs\in\Aut(G),\quad b\in G. $$
If $K=\pi J$, for some $\pi$ in the affine group, we say 
$J$ and $K$ are {\bf affinely equivalent}.

\begin{lemma}
\label{harmonicaffequiv}
If $J$ and $K$ are subsets of a finite abelian group $G$,
which are translates of each other,
then the harmonic frames $\Phi_J$ and $\Phi_K$ are 
projectively unitarily equivalent.
\end{lemma}

\begin{proof}
Suppose $K=J-b$. Since $\Phi_J=(\xi|_J)_{\xi\in\hat G}$, we need to show
$$ \xi|_K = c_\xi U (\xi|_J), \qquad \xi\in\hat G, $$
where $U:\C^J\to\C^K$ is unitary. Let $U_b:\C^J\to\C^K$ be the unitary map
$$ (U_b v )(k) := v(k+b), \qquad k\in K. $$
Since $\xi$ is a character, we have
$$ ( U_b \xi|_J )(k) 
= \xi|_J(k+b)
= \xi(k+b)
= \xi(k)\xi(b)
= \xi|_K(k)\xi(b), $$
and so we can take $U=U_b$ and $c_\xi =1/ \xi(b)$.
\end{proof}

The converse: that projective unitary equivalence 
implies $J$ and $K$ are translates of each other appears to be true.

\begin{theorem}
\label{harmonicequiv}
Suppose $J$ and $K$ are subsets of a finite abelian group $G$. Then 
\begin{enumerate} 
\item If $J$ and $K$ are translates, 
then $\Phi_J$ and $\Phi_K$ are projectively unitarily equivalent.
\item If $J$ and $K$ are multiplicatively equivalent, 
then $\Phi_J$ and $\Phi_K$ are unitarily equivalent after reindexing
by an automorphism.
\item If $J$ and $K$ are affinely equivalent, 
then $\Phi_J$ and $\Phi_K$ are projectively unitarily equivalent after 
reindexing by an automorphism.
\end{enumerate}
\end{theorem}

\begin{proof} 
The first part is Lemma
\ref{harmonicaffequiv}, the second is given in \cite{CW11} (Theorem 3.5),
and third follows by combining the first two.
\ref{harmonicaffequiv}.
\end{proof}

\begin{example} 
\label{complementframesequiv}
Let $p>2$ be a prime. Then all harmonic
frames of $p$ vectors in $\C^2$ are projectively 
unitarily equivalent up to reindexing
(to $p$ equally spaced vectors in $\Rd$).
This follows since there is a unique affine map, taking a
sequence of two distinct elements of $\Z_p$ to any other.
In particular, allowing for reindexing, 
the two harmonic frames of three vectors in $\C^2$
which are unitarily inequivalent (one is real, one is complex)  are 
projectively unitarily equivalent.
\end{example}

The conditions of 1,2,3 
of Theorem \ref{harmonicequiv}
say that $J$ and $K$ are in the same orbit
under action of the {\em group of translations}, the 
{\em automorphism group},
and the {\em affine group} on the subsets of $G$, respectively.
Using this, we were able to calculate the various equivalences using
the computer algebra package MAGMA \cite{BCP97}.
The results of these calculations for the cyclic harmonic
frames are summarised in Fig.\ 5.
These indicate that the number of projective unitary equivalence classes
is much smaller than the number of unitary equivalence classes 
(up to any reindexing).
There are just a few cases where the number of equivalence classes is
smaller than that predicted by the group theoretic calculations,
because there is a reindexing which is {\em not} an automorphism
which makes harmonic frames equivalent. 
This was previously observed in the case of unitary equivalence 
\cite{CW11}.
In these cases the larger group theoretic estimate is given 
in the row below in Fig.\ 5.

\begin{figure}[h!]
\centering
\hbox{\hskip3.6truecm $d=2$\hskip3.1truecm $d=3$ \hskip3.1truecm $d=4$}
\noindent\begin{tabular}[t]{|c|c|c|}
\hline
n & uni & proj \\ \hline
2 & 1 & 1 \\ \hline
3 & 2 & 1 \\ \hline
4 & 3 & 2 \\ \hline
5 & 3 & 1 \\ \hline
6 & 6 & 3 \\ \hline
7 & 4 & 1 \\ \hline
8 & 7 & 3 \\ \hline
9 & 6 & 2 \\ \hline
10 & 9 & 3 \\ \hline
11 & 6 & 1 \\ \hline
12 & 13 & 5 \\ \hline
13 & 7 & 1 \\ \hline
14 & 12 & 3 \\ \hline
15 & 13 & 3 \\ \hline
\end{tabular}
\hskip1truecm
\begin{tabular}[t]{|c|c|c|}
\hline
n & uni & proj \\ \hline
3 & 1 & 1 \\ \hline
4 & 3 & 1 \\ \hline
5 & 3 & 1 \\ \hline
6 & 11 & 3 \\ \hline
7 & 7 & 2 \\ \hline
8 & 16 & 4 \\ 
 & 17 &  \\ \hline
9 & 15 & 3 \\ \hline
10 & 29 & 4 \\ \hline
11 & 17 & 2 \\ \hline
12 & 56 & 9 \\ 
 & 57 &  \\ \hline
13 & 25 & 3 \\ \hline
\end{tabular}
\hskip1truecm
\begin{tabular}[t]{|c|c|c|}
\hline
n & uni & proj \\ \hline
5 & 2 & 1 \\ \hline
6 & 9 & 3 \\ \hline
7 & 7 & 2 \\ \hline
8 & 21 & 6 \\ 
 & 23 & 5 \\ \hline
9 & 23 & 4  \\ 
 & 24 &  \\ \hline
10 & 53 & 9 \\ 
 & 54 &  \\ \hline
11 & 34 & 4\\ \hline
12 & 138 & 21 \\ 
 & 141 &  \\ \hline
\end{tabular} \\
\vskip0.5truecm

\hbox{\hskip3.6truecm $d=5$\hskip3.1truecm $d=6$ \hskip3.1truecm $d=7$}
\noindent\begin{tabular}[t]{|c|c|c|}
\hline
n & uni & proj \\ \hline
5 & 1 & 1 \\ \hline
6 & 4 & 1 \\ \hline
7 & 4 & 1 \\ \hline
8 & 19 & 4 \\ 
 & 20 &  \\ \hline
9 & 23 & 4 \\ 
 & 24 &  \\ \hline
10 & 67 & 9 \\ \hline
11 & 48 & 6 \\ \hline
\end{tabular}
\hskip1truecm
\begin{tabular}[t]{|c|c|c|}
\hline
n & uni & proj \\ \hline
6 & 1 & 1 \\ \hline
7 & 2 & 1 \\ \hline
8 & 11 & 3 \\ \hline
9 & 16 & 3 \\ \hline
10 & 55 & 9  \\ 
 & 56 &  \\ \hline
11 & 48 & 6\\ \hline
\end{tabular}
\hskip1truecm
\begin{tabular}[t]{|c|c|c|}
\hline
n & uni & proj \\ \hline
7 & 1 & 1 \\ \hline
8 & 4 & 1 \\ \hline
9 & 8 & 2 \\ \hline
10 & 32 & 4 \\ \hline
11 & 34 & 4 \\ \hline
12 & 228 & 25 \\
 & 234 & \\ \hline
\end{tabular}
\caption{The number of unitary and projective unitary equivalence classes
(up to reindexing) of cyclic harmonic frames of $n$ vectors in 
$\Cd$, $d=2,\ldots,7$. When the group theoretic estimate 
of Theorem \ref{harmonicequiv}
is larger
(because there are reindexings which are not automorphisms) it is given
in the row below.}
\end{figure}

\eject

\end{document}